\documentclass[12pt]{article}
\usepackage{amsmath, amsfonts, amsthm, amssymb, verbatim}
\usepackage{graphicx}
\usepackage[mathcal]{euscript}
\usepackage{amstext}
\usepackage{amssymb,mathrsfs}
\usepackage{bbm}
\usepackage[latin1]{inputenc}
\newcommand{\R}{\mathbb R}
 
 \newcommand{\N}{\mathbb N}

  \newcommand{\E}{\mathbb E}

\newcommand{\ve}{\varepsilon}

\newcommand{\dive}{{\rm div}}

\newcommand{\I}{\int}

\newtheorem{theorem}{Theorem}[section]
\newtheorem{proposition}[theorem]{Proposition}
 \newtheorem{remark}[theorem]{Remark}

\newtheorem{definition}[theorem]{Definition}
\newtheorem{hypothesis}[theorem]{Hypothesis}

\begin{document}

\title{ Well-posedness  of  the  stochastic transport   equation with 
 unbounded drift.   }

\author{David A.C. Mollinedo \footnote{Universidade Tecnol\'{o}gica Federal do Parana, Brazil. E-mail:  {\sl davida@utfpr.edu.br}
}, 
Christian Olivera\footnote{Departamento de Matem\'{a}tica, Universidade Estadual de Campinas, Brazil. 
E-mail:  {\sl  colivera@ime.unicamp.br}.
}}

\date{}

\maketitle

\textit{Key words and phrases. 
Stochastic partial differential equation, transport equation, Low regularity, Stochastic characteristic method.}

\vspace{0.3cm} \noindent {\bf MSC2010 subject classification:} 60H15, 
 35R60, 
 35F10, 
 60H30. 


%
\begin{abstract}
The Cauchy problem for a multidimensional linear transport
equation with unbounded drift is investigated.
Provided the  drift is  Holder continuous , existence, uniqueness
and strong  stability of solutions are obtained. The proofs are  based on a careful analysis of the associated stochastic flow of characteristics and  techniques of stochastic analysis.  
 \end{abstract}
%
\maketitle

%

\section {Introduction} \label{Intro}

We consider the deterministic linear transport equation in $ \mathbb{R}^{d}$

\begin{equation}\label{trasports}
    \partial_t u(t, x) +  b(t,x) \cdot  \nabla u(t,x)  = 0 \, ,
\end{equation}

this equation is one of the most fundamental and at the same time most
elementary partial differential equation with applications in a wide range of problems
from physics, engineering, biology or social science. 

 When the coefficients are regular  the unique solution is found  by the method of characteristics.  
Recently research activity has been devoted to study continuity/transport equations with rough coefficients, showing a well-posedness result.
Di Perna andLions   \cite{DL}  have introduced the notion of renormalized solution  to transport equation :
it is a solution such that

\begin{equation}\label{renord}
    \partial_t \beta (u(t, x) )+  b(t,x) \nabla \beta ( u(t,x))   = 0 .
\end{equation}

for any suitable non-linearity $\beta$. Notice that (\ref{renord}) holds for smooth solutions, by
an immediate application of the chain-rule. However, when the vector field is not
smooth, we cannot expect any regularity of the solutions, so that (\ref{renord}) is a nontrivial
request when made for all bounded distributional solutions. The renormalization property asserts that nonlinear compositions of
the solution are again solutions, or alternatively that the chain-rule holds in this weak
context. The overall result which motivates this definition is that, if the renormalization
property holds, then solutions of (\ref{trasports}) are unique and stable.

 In the case $b$ has  $W^{1,1}$ 
spatial regularity
(together with a condition of boundedness on the divergence) 
  the commutator lemma between smoothing convolution and weak solution  can be
proved and, as a consequence, all $L^{\infty}$-weak solutions are renormalized. Afters some intermediate results  (
\cite{Colomb1}  and  \cite{Colomb2}  )
the theory has been generalized
by    L. Ambrosio \cite{ambrisio} to the case of only $BV$  regularity for b instead of   $W^{1,1}$.
See  \cite{ambrisio2} and \cite{lellis}    for a 
nice review on that.

We consider the following stochastic transport equation 
 
\begin{equation}\label{trasport}
 \left \{
\begin{aligned}
    &\partial_t u(t, x) +  \,   \, \big(b(t, x) + \frac{d B_{t}}{dt} \big) \nabla u(t, x) = 0,
    \\[5pt]
    &u|_{t=0}=  u_{0},
\end{aligned}
\right .
\end{equation}
$\big( (t,x) \in U_T, \omega \in \Omega \big)$, where $U_T= [0,T] \times \R^d$, for
 $T>0$ be any fixed real number, $(d \in \N)$, $b:\R_+ \times
\R^d \to \R^d$ is a given vector field,  
$B_{t} = (B_{t}^{1},...,B _{t}^{d} )$ is a
standard Brownian motion in $\mathbb{R}^{d}$ and the stochastic
integration is taken (unless otherwise mentioned) in the Stratonovich sense. 

\medskip

\bigskip
The Cauchy problem for the stochastic transport/continuity  equation has taken great attention recently.  F. Flandoli, M. Gubinelli and  E. Priola
in \cite{FGP2}  obtained well-posedness of the stochastic problem for a bounded  H\"older  continuous drift term, with some integrability conditions on the divergence.  E.Fedrizzi and F. Flandoli  in \cite{Fre1}  obtained a well-posedness result , in the class of   local Sobolev solutions,  under only some integrability conditions on the drift. There, it is only assumed that
\begin{equation}\label{LPSC}
    \begin{aligned} 
    &b\in L^{q}\big( [0,T] ; L^{p}(\mathbb{R}^{d}) \big) \, , \\[5pt]
 \mathrm{for}  \qquad &\  p,q \in [2,\infty) \, , \qquad   \qquad  \frac{d}{p} + \frac{2}{q} < 1 \, .
\end{aligned}
\end{equation}
The well-posedness of the Cauchy problem \eqref{trasport} under condition \eqref{LPSC}
for measurable initial condition was also  considered in \cite{NO} but with free divergence condition. In \cite{Beck}, using a technique based on the regularising effect 
observed on expected values of moments of the solution, well-posedness of \eqref{trasport} was obtained also for the limit cases of $p,q=\infty$ or when the inequality in \eqref{LPSC} becomes an equality. In \cite{Fre2}  the authors proved  uniqueness when the field   vectors $b\in L_{loc}^{2}$ for a new class of  solutions. Finally , we mention the paper \cite{Moli} where the authors show uniqueness of the
one-dimensional  continuity equations when  the drift is measurable and  the linear growth.

\medskip

 In this  paper we obtain a well-posedness result for unbounded vector fields. We prove  existence, uniqueness   and strong stability result   for   $W^{1,p}-$ solutions with  unbounded locally Holder continuous drift and $\dive b\in L^{q}_{loc} $
 . The proofs are  based on a careful analysis of the associated stochastic flow of characteristics and stochastic calculus techniques. 
The uniqueness  result is an improvement of the condition considered  in the seminar paper \cite{FGP2} where   case  bounded locally Holder drift was considered.

 We  would like to point that in the deterministic transport equation it is not possible to 
show existence in the class of  $W^{1,p}-$ solutions. As showed by Colombini, Luo and Rauch in \cite{Colom}, 
there exists an important example of $b \in L^\infty \cap W^{1,p}, (\forall p < \infty)$, such that
the  propagation of the continuity in the deterministic transport equation is missing. That is to say,
 one may start  with a continuous initial data, but the deterministic solution of the transport equation is not continuous.
However, in the stochastic case we have the persistence property.

\medskip

In fact, through of this paper, we fix a stochastic basis with a
$d$-dimensional Brownian motion $\big( \Omega, \mathcal{F}, \{
\mathcal{F}_t: t \in [0,T] \}, \mathbb{P}, (B_{t}) \big)$.

\subsection{Hypothesis}

 Throughout   this article we consider $1< p < \infty$. 
 We assume  

\begin{hypothesis}\label{hyp}
\begin{equation}\label{con1}
    b\in C^{\theta}(\R^d,\R^d)\,,
\end{equation}
and
\begin{equation}\label{con2}
 \dive \, b(x) \in L_{loc}^{q} \big( \R^d \big) \ with \  \frac{1}{p} +  \frac{1}{q}=1 .
\end{equation}
Moreover, the initial condition is taken to be
\begin{equation*}\label{conIC}
 u_0 \in  W^{1,2p}(\R^d)   \, .
\end{equation*}
\end{hypothesis}

\section{Results.}

\subsection{Notations}

For any  $\theta \in (0,1)$, we denoted   $C^{\theta}(\R^d;\R^d)$, $d \geq 1$ the space of the   field vectors  $f:\R^d \to \R^d$  such
that  
\begin{align*}
[f]_{\theta}:=\sup_{x\neq y, |x-y|\leq 1} \frac{|f(x)-f(y)|}{|x-y|^\theta} < \infty\,.
\end{align*}

 The space $C^{\theta}(\R^d;\R^d)$ is the Banach space with the norm

\begin{align}\label{eq-t2-norma para C theta}
\|f\|_{\theta}=\|(1+|\cdot|)^{-1}f\|_\infty+[f]_{\theta}. 
\end{align}

\medskip

Let us start by setting the notation used and then recalling the main results. For $0\leq s\leq t$ and $x\in\mathbb{R}^{d}$, consider the following
stochastic differential equation in $\mathbb{R}^{d}$%

\begin{equation}
\label{itoass}X_{s,t}(x)= x + \int_{s}^{t} b( X_{s,r}(x)) \ dr + B_{t}-B_{s},
\end{equation}

\noindent where $X_{s,t}(x)= X(s,t,x)$, also $X_{t}(x)= X(0,t,x)$.   Under condition (\ref{con1}), $X_{s,t}(x)$ is a
stochastic flow of $C^{1}$-diffeomorphism( see \cite{FGP}) Moreover, the inverse $Y_{s,t}(x):=X_{s,t}^{-1}(x)$ satisfies the
following backward stochastic differential equation%

\begin{equation}
\label{itoassBac}Y_{s,t}(x)= x - \int_{s}^{t} b( Y_{r,t}(x)) \ dr - (B_{t}-B_{s}),
\end{equation}

\noindent for $0\leq s\leq t$, see \cite{FGP2}. We denote  by $\phi_{s,t}$ the flow associeted to $X_{s,t}$ and $\psi_{s,t}$ its inverse.

We  also recall   the important results  in \cite{FGP} : Let     $b_n\in C^{\theta}(\R^d,\R^d)$,
 and let $\phi_{s,t}^{n}$ be the corresponding stochastic flows, assume that $\|b_n -b\|_{C_b^{\theta}(\R^d,\R^d)
}\rightarrow 0 $ as $n\rightarrow \infty$, then for all $p\geq 1$ we have

\begin{equation}\label{est1}
 \lim_{n\rightarrow\infty}  \sup_{x \in \R^d}   \sup_{s\in[0,T]} \mathbb{E}[  \sup_{t \in [s,T]}  |D\phi_{s,t}^{n}(x)- D\phi_{s,t}(x)|^p] =0 ,
\end{equation}

\begin{equation}\label{est2}
 \lim_{n\rightarrow\infty}\sup_{x \in \R^d} \sup_{s \in [0,T]}\mathbb{E}[  \sup_{t \in [s,T]}  |\frac{\phi_{s,t}^{n}(x)- \phi_{s,t}(x)}{ 1+ |x|}|^p] =0 ,
\end{equation}

and

\begin{equation}\label{est3}
 \sup_{n}  \sup_{x \in \R^d}   \sup_{s\in[0,T]} \mathbb{E}[  \sup_{t \in [s,T]}  |D\phi_{s,t}^{n}(x)|^p] <  \infty.
\end{equation}

The same results are valid for the backward flows $\psi_{s,t}^{n}$ and  $\psi_{s,t}$ since 
are solutions of the same SDE driven by the drifts $-b_{n}$ and $-b$.

\subsection{Definition.}

The next definition tells us in which sense a stochastic process is a $W^{1,p}$-solution of \eqref{trasport}. We denoted $ C_0^{\infty}(\R^d)$
the space of the  test functions with compact support.  We  denoted $\mu=e^{-|x|^{2}}$  the gaussian measure in $\R^{d}$. 

\begin{definition}\label{defisolu1}
 A stochastic process $u\in  L^{2p}( \Omega\times[0, T] \times \mathbb{R}^{d}) \cap L^{p}( \Omega\times[0, T], W^{1,p}(\mathbb{R}^{d}),\mu)$ is called a  $W^{1,p}$- weak solution of the Cauchy problem \eqref{trasport} when:
 for any $\varphi \in C_0^{\infty}(\R^d)$, the real valued process $\int  u(t,x)\varphi(x)  dx$ has a continuous modification which is an
$\mathcal{F}_{t}$-semimartingale, and for all $t \in [0,T]$, we have $\mathbb{P}$-almost surely
\begin{equation} \label{DISTINTSTR}
\begin{aligned}
    \int_{\R^d} u(t,x) \varphi(x) dx = &\int_{\R^d} u_{0}(x) \varphi(x) \ dx
   -\int_{0}^{t} \!\! \int_{\R^d}     \, b^{i}(x) \partial_i u(s,x) \varphi(x) \ dx ds
\\[5pt]
    &+ \int_{0}^{t} \!\! \int_{\R^d} u(s,x) \ \partial_{i} \varphi(x) \ dx \, {\circ}{dB^i_s} \, .
\end{aligned}
\end{equation}
\end{definition}

\bigskip

\begin{remark}\label{lemmaito}  Using the same idea as in Lemma 13 \cite{FGP2}, on can write the problem (\ref{trasport}) in
 It\^o form as follows. A  stochastic process   $u\in  L^{2p}( \Omega\times[0, T] \times \mathbb{R}^{d}) \cap L^{p}( \Omega\times[0, T], W^{1,p}(\mathbb{R}^{d}),\mu)$
is  a weak $W^{1,p}-$solution  of the SPDE (\ref{trasport}) iff for every test function $\varphi \in C_{0}^{\infty}(\mathbb{R}^{d})$, the process $\int u(t, x)\varphi(x) dx$ has a continuous modification which is a $\mathcal{F}_{t}$-semimartingale and satisfies the following It\^o' formulation

\[
\int u(t,x) \varphi(x) dx= \int u_{0}(x) \varphi(x) \ dx
\]
\[
-\int_{0}^{t} \int   b^{i}(x) \partial_i u(s,x)   \varphi(x)  \ dx\, ds
\]
\[
+  \int_{0}^{t} \int  \partial_{i}\varphi(x) u(s,x) \ dx
\ dB_{s}^{i} +\frac{1}{2}   \int_{0}^{t} \int  \Delta \,\varphi(x) u(s,x) \ dx\, ds  . 
\]
\end{remark}

\subsection{   Existence.}

\begin{proposition}\label{renor}  We assume  hypothesis (\ref{hyp}). Then $u(t,x)=u_{0}(X_t^{-1})$ is a  $W^{1,p}-$solution of the equation   \eqref{trasport}.
\end{proposition}
\begin{proof}

{\it Step 1 : Regular initial data.} We assume that $u_0 \in C_0^{\infty}(\R^d)$. Let $\{\rho_\varepsilon\}_\varepsilon$ be a family of standard symmetric mollifiers.
  We define the family of regularised coefficients as
 $b^{\epsilon}(x) = (b \ast \rho_\varepsilon) (x)  $.
 
For any fixed $\varepsilon>0$, the classical theory of Kunita, see \cite{Ku} or \cite{Ku3},
 provides the existence of a unique 
  solution $u^{\varepsilon}$ to the regularised equation 
	
	\begin{equation}\label{trasport-reg}
 \left \{
\begin{aligned}
    &d u^\varepsilon (t, x) +  \nabla u^\varepsilon (t, x)  \cdot \big( b^\varepsilon ( x)  dt +
 \circ d B_{t} \big) = 0 \, ,
    \\[5pt]
    &u^\varepsilon \big|_{t=0}=  u_{0}
\end{aligned}
\right .
\end{equation}

in terms of the (regularised) initial condition and the inverse flow $(\phi_t^\varepsilon)^{-1}$ associated to the equation of characteristics of \eqref{trasport-reg}, which reads 
\begin{equation*}
d X_t = b^\varepsilon ( X_t) \, dt + d B_t \, ,  \hspace{1cm}   X_0 = x \,.
\end{equation*}

We denoted $(\phi_t^\varepsilon)^{-1}$  by $\psi^\epsilon$. If $u^\varepsilon$ is a solution of \eqref{trasport-reg}, it is also a weak solution, which means that for any test function $\varphi \in C_c^\infty(\R^d)$, $u^\varepsilon$ satisfies 
the following equation (written in It\^o form)
\begin{equation}\label{transintegralR2}
\begin{aligned}
    \int_{\R^d} u^\varepsilon(t,x) &\varphi(x) \, dx= \int_{\R^d} u_{0}(x) \varphi(x) \, dx
   - \int_{0}^{t} \!\! \int_{\R^d} \partial_i u^\varepsilon(s,x) \, b^{i, \varepsilon}(x)  \varphi(x) \, dx ds
\\[5pt]
    &+ \int_{0}^{t} \!\! \int_{\R^d} u^\varepsilon(s,x) \, \partial_{x_i} \varphi(x) \, dx \, dB^i_s
    + \frac{1}{2} \int_{0}^{t} \!\!\int_{\R^d} u^\varepsilon(s,x) \Delta \varphi(x) \, dx ds \, .
\end{aligned}
\end{equation}

 We claim that 

\begin{equation}\label{eq-t2-conv do u espilon con fluxo epsilon}
u_0 (\psi_t^{\epsilon}) \rightarrow u_0 (\psi_t)\,,\quad\mathbb{P}\otimes dt \otimes dx-\text{a.e}\,\,\,\,\, \Omega\times[0,T]\times B_R\,,
\end{equation}

as  $\epsilon \rightarrow 0$. In fact, we have 

\[
\lim_{\epsilon\rightarrow 0}  \I_{0}^{T} \I_{B_R} \I_{\Omega}|\psi_{t}^{\epsilon}(x)-\psi_{t}(x)|^p\,\mathbb{P}(dw)\,dx\,dt 
\]
\[
 \leq (1+R)^p\lim_{\epsilon\rightarrow 0} \I_0^T \I_{B_R}\I_{\Omega}\frac{|\psi_t^\epsilon(x)-\psi_t(x)|^p}{(1+|x|)^p}\mathbb{P}(dw)\,dx\,dt 
\]

\[
\leq T |B_R| (1+R)^p \lim_{\epsilon\rightarrow 0} \sup_{t\in [0,T]} \sup_{x\in B_R}\E\bigg[\frac{|\psi_t^\epsilon(x)-\psi_t(x)|^p}{(1+|x|)^p}\bigg] =0.
\]

Therefore, the sequence  $\{\psi^\epsilon\}$ converge  to $\psi$ in  $L^p(\Omega\times[0,T]\times B_R)$  as  $\epsilon \rightarrow 0$. Thus 
there exist a subsequence  $\psi^{\epsilon}$, such that 
\begin{equation}
\psi^\epsilon\rightarrow \psi\,,\quad\mathbb{P}\otimes dt \otimes dx-\text{a.e.}\,\,\,\,\, \Omega\times[0,T]\times B_R\,,
\end{equation}

Now, we have 

\[
\int_{0}^{t} \int_{\R^d} \partial_i u^\epsilon(s,x) \ b^{i, \epsilon}(x) \ \varphi(x) \ dx \,ds 
=
\]
\[
- \int_{0}^{t} \int_{\R^d} u^\epsilon(s,x) \ b^{i, \epsilon}(x) \ \partial _i \varphi(x) \ dx \,ds 
- \int_{0}^{t} \int_{\R^d} u^\epsilon(s,x) \ div  b^{ \epsilon}(x) \  \varphi(x) \ dx \,ds .
\]

By  dominated convergence  we obtain

\[
 \int_{0}^{t} \int_{\R^d} u^\epsilon(s,x) \ b^{i, \epsilon}(x) \ \partial _i \varphi(x) \ dx \,ds 
\rightarrow  \int_{0}^{t} \int_{\R^d} u(s,x) \ b^{i}(x) \ \partial _i \varphi(x) \ dx \,ds,
\]

\[
 \int_{0}^{t} \int_{\R^d} u^\epsilon(s,x) \ div  b^{ \epsilon}(x) \  \varphi(x) \ dx \,ds 
\rightarrow   \int_{0}^{t} \int_{\R^d} u(s,x) \ div b(x) \  \varphi(x) \ dx \,ds, 
\]

\[
\frac{1}{2} \int_{0}^{t} \int_{\R^d} u^\epsilon(s,x) \Delta\varphi(x)\,dx\, ds  \rightarrow \frac{1}{2}
 \int_{0}^{t} \int_{\R^d} u(s,x) \Delta\varphi(x)\,dx\, ds\,.
\]

By stochastic  dominated convergence theorem we obtain 

\[
\int_{0}^{t} \int_{\R^d} u^\epsilon(s,x) \ \partial_{i} \varphi(x) \ dx \,{dB^i_s}\rightarrow 
\int_{0}^{t} \int_{\R^d} u(s,x) \ \partial_{i} \varphi(x) \ dx \,{dB^i_s}.
\]

Taking the limit in equation (\ref{transintegralR2})    we conclude that
$u(t,x)=u_{0}(X_{t}^{-1})$  verifies

\begin{align}\label{eq-t1-regul renormalizada com beta-itobis}
\int_{\R^d} u(t,x) \varphi(x) dx   & =\int_{\R^d} u_0(x)\varphi(x) \ dx \nonumber\\[5pt]
                                                   & -\int_{0}^{t} \int_{\R^d} \partial_i u(s,x) \  b^{i}(x) \ \varphi(x) \ dx \,ds \nonumber\\[5pt]
                                                                                  	& + \int_{0}^{t} \int_{\R^d} u(s,x) \ \partial_{i} \varphi(x) \ dx \,{dB^i_s} \nonumber\\[5pt]
                                                   & + \frac{1}{2} \int_{0}^{t} \int_{\R^d} u(s,x) \Delta\varphi(x)\,dx\, ds\,.
\end{align}

Moreover, we have 

\[
\E \int  |  u_0 (\psi_t) |^{2p} \ dx =  \int \E \left[ J\phi_t  \right]  |u_0 (x) |^{2p} \ dx \leq C ,
\]

 and

\[
\E \int  |  D u_0 (\psi_t) |^{p} \  e^{-|x|^{2}} \ dx = \E \int  |  Du_0 (\psi_t)|^{p} \  |D\psi_t   |^{p} e^{-|x|^{2}} \ dx  
\]
\[
\leq   C  ( \E \int  |  D u_0 (\psi_t) |^{2p}  \ e^{-|x|^{2}} dx +  \E \int   |D\psi_t   |^{2p}  e^{-|x|^{2}} dx )
\]

\[
\leq    \E \int  |  D u_0 (x) |^{2p}  \ | J\phi_t| \ e^{-|\phi_t|^{2}} dx      +   C 
\]

\[
\leq  C  \ \E \int  |  D u_0 (x) |^{2p}   dx      +   \  C 
\]

\noindent where we used  (\ref{est1}) and  (\ref{est3}). Thus we conclude $u(t,x)=u_0 (\psi_t)$ that is a 
is a $W^{1,p}$-solution of the Cauchy problem (\ref{trasport}).

\medskip

{\it Step 2: Irregular initial data.}   $u_0 \in  W^{2,p}(\R^d)$. Let $\{\rho_\varepsilon\}_\varepsilon$ be a family of standard symmetric mollifiers.
 Consider a  nonnegative smooth cut-off function $\eta$ supported on the ball of radius 2 and such that
 $\eta=1$ on the ball of radius 1. For every $\varepsilon>0$ introduce the rescaled functions
 $\eta_\varepsilon (\cdot) =  \eta(\varepsilon \cdot)$.  We define the family of regular approximations of the initial condition $u_0^\varepsilon (x) = \eta_\varepsilon(x) \big( [ u_0(\cdot) \ast \rho_\varepsilon (\cdot) ] (x)  \big)$. By the   step 1 we have that 
$u^\varepsilon(t,x) =u_0^\varepsilon(\psi_t)$ verifies

\begin{align}\label{eqirr}
\int_{\R^d} u^\epsilon(t,x) \varphi(x) dx   & =\int_{\R^d} u_0^\varepsilon(x)\varphi(x) \ dx \nonumber\\[5pt]
                                                   & -\int_{0}^{t} \int_{\R^d} \partial_i u^\epsilon(s,x) \ b(x) \ \varphi(x) \ dx \,ds \nonumber\\[5pt]
                                               & + \int_{0}^{t} \int_{\R^d} u^\epsilon(s,x) \ \partial_{i} \varphi(x) \ dx \,{dB^i_s} \nonumber\\[5pt]
                                                   & + \frac{1}{2} \int_{0}^{t} \int_{\R^d} u^\epsilon(s,x) \Delta\varphi(x)\,dx\, ds\,.
\end{align}

 We claim that 

\begin{equation}\label{conver}
  u_0^{^\varepsilon}(\psi)  \rightarrow    u_0(\psi)       \ in \,\,\,\, L^{2p}(\Omega\times[0,T]\times  \R^d) \,,
\end{equation}

In fact,  doing the variable change $x=\psi_t$  we have

$$
\int_{0}^{T} \int_{\R^d} E |u^\epsilon(s,x) -  u(s,x))   |^{2p}\ dx \ ds
$$
$$
=   \int_{0}^{T} \int_{\R^d}   |u_0^\epsilon(x) -  u_0(x)|^{2p}   E | J\phi_t |\ dx \ ds
$$
$$
\leq  C \int_{0}^{T} \int_{\R^d}   |u_0^\epsilon(x) -  u_0(x)|^{2p}   dx \ ds
$$
$$
\leq  C  \int_{\R^d}   |u_0^\epsilon(x) -  u_0(x)|^{2p}   dx \ ds .
$$

where  we used  (\ref{est1}) and (\ref{est3}).

Now, we observe that 

\[
\E \int  |  D u_0 (\psi_t) |^{p} \  e^{-|x|^{2}} \ dx = \E \int  |  Du_0 (\psi_t)|^{p} \  |D\psi_t   |^{p} e^{-|x|^{2}} \ dx  
\]
\[
\leq   C  ( \E \int  |  D u_0 (\psi_t) |^{2p}  \ e^{-|x|^{2}} dx +  \E \int   |D\psi_t   |^{2p}  e^{-|x|^{2}} dx )
\]

\[
\leq    \E \int  |  D u_0 (x) |^{2p}  \  | J\phi_t|  \ e^{-|\phi_t|^{2}} dx      +   C 
\]

\[
\leq  C  \ \E \int  |  D u_0 (x) |^{2p}   dx      +   \  C 
\]

where  we used (\ref{est1}) and  (\ref{est3}). 

 Then  passing to the limit in (\ref{eqirr}) we
 conclude the that $u(t,x)=u_0 (\psi_t)$ is a $W^{1,p}-$solution
of equation (\ref{trasport}).

\end{proof}

\subsection{Uniqueness.}
\label{UNIQUE}

In this section we  prove the uniqueness result for $W^{1,p}$- solutions.  
\bigskip

\begin{proposition}\label{uni} Assume hypothesis (\ref{hyp}).  
If $u_1,u_2 \in L^p([0,T] \times \Omega, W^{1,p}(\mathbb{R}^{d}))$ are two weak  $W^{1,p}-$solutions 
for the Cauchy problem \eqref{trasport}, with the same initial data 
$u_{0}\in W^{1,2p}(\mathbb{R}^{d})$, then  $u_1= u_2$ almost everywhere 
in $[0,T] \times\R^d \times \Omega$. 
\end{proposition}

\begin{proof} By linearity, it is enough to show that a weak
$W^{1,p}-$solution $u$ with initial condition $u_{0}(x)=0$ vanishes
identically.   Let $\phi_{\varepsilon}, \phi_{\delta}$ be standard symmetric mollifiers. Thus $u_{\varepsilon}(t,\cdot)=u(t,\cdot)\ast \phi_{\varepsilon}$
verifies
\begin{equation} \label{DISTINTSTRAPP}
\begin{aligned}
    \int_{\R^d} u_\varepsilon(t,z) dz &= -\int_{0}^{t} \!\! \int_{\R^d} \partial_i u(s,z) \ b^i(z)  \phi_\varepsilon(y-z) \ dz ds
\\[5pt]
    &\;  + \int_{0}^{t} \!\! \int_{\R^d} u(s,z) \ \partial_{i} \phi_\varepsilon(y-z) \ dz \, {\circ}{dB^i_s}.
\end{aligned}
\end{equation}

One remarks that, for each $\ve> 0$ the equation for $u_\ve$ is strong in the analytic sense.

\medskip
Now,   we denote by $b^{\delta}$ the standard mollification of $b$,
and let   $X_t^{\delta }$  be  the associated flow given by  the SDE (\ref{itoass}) replacing  $b$ by 
 $b^{\delta}$.  
Similarly, we consider $Y^\delta_{t}$, which satisfies the backward SDE \eqref{itoassBac}.

\medskip
 Doing the variable $y=X_t^{\delta}$ we have 
\begin{equation}
\label{PUSHFORWARD}
    \int_{\R^d}  u_{\varepsilon}(t, X_t^{\delta}) \; \varphi(y) \ dy 
    =\int_{\R^d}  u_{\varepsilon}(t,y) \;    \ JY_t^{\delta} \varphi(Y_t^{\delta}) \ dy ,
\end{equation}

for each $t \in [0,T]$. 

Now, we observe that  by  It\^{o} formula or by Kunita \cite{Ku3} $ v^{\delta}(t,x)=JY^{\delta}_t  \ \varphi(Y_t^{\delta}),$ satisfies the continuity  equation in the classical sense, that is, it satisfies 

\begin{equation}\label{trasportC}
 \left \{
\begin{aligned}
    & dv^{\delta}(t, x) + Div (b^{\delta}(x)  v^{\delta}(t, x) )  \ dt + \nabla v^{\delta}(t, x)  \circ  d B_{t}= 0,
    \\[5pt]
    & v^{\delta}|_{t=0}=  \varphi(x).
\end{aligned}
\right .
\end{equation}

 Consider a  nonnegative smooth cut-off function $\eta$ supported on the ball of radius 2 and such that
 $\eta=1$ on the ball of radius 1. For each $R>0$ introduce the rescaled functions.
 $\eta_R (\cdot) =  \eta(\frac{1}{R} \cdot)$

 Then we may apply It\^o's formula to the product of two semimartingales  

$$u_{\varepsilon}(t,x)  v^{\delta}(t, x),$$ 
and obtain that

\begin{equation}
\label{UNIQ10}
  \begin{aligned}
  \int_{\R^d}   \eta_R (y) u_{\varepsilon}(t,y) & \;  v^{\delta}(t,y) \ dy  
     =-\int_{0}^{t}\!\!  \int_{\R^d}  \eta_R (y)  u_\varepsilon(s,y)  \; div ( b^{\delta}(y) v_s^{\delta}(s,y))   \ dy  \  ds 
\\[5pt]
&- \int_{0}^{t}  \!\! \int_{\R^d} \eta_R (y)  u_\varepsilon(s,y) \;   \partial_{i} [v^{\delta}(s,y)] dy \  \circ dB_{s}^{i} 
\\[5pt]
 &-\int_{0}^{t} \!\! \int_{\R^d} \eta_R (y)  v^{\delta}(s,y)    \int_{\R^d}   \partial_i u(s,z) \; b^{i}(z) \  \phi_{\varepsilon}(y-z) \ dz \ dy \ ds 
\\[5pt]
& +  \int_{0}^{t} \!\! \int_{\R^d} \eta_R (y)  v^{\delta}(s,y)  \int_{\R^d} 
 u(s,z) \; \partial_{i} \phi_{\varepsilon}(y-z) \ dz \ dy \  \circ dB_{s}^{i}. 
\end{aligned}
\end{equation}

We observe  that 

\[
 \int_{0}^{t} \!\! \int_{\R^d} \eta_R (y)  v^{\delta}(s,y)  \int_{\R^d} u(s,z) \; \partial_{i} \phi_{\varepsilon}(y-z) \ dz \ dy \  \circ dB_{s}^{i}.
\]
\[
 = \int_{0}^{t} \!\! \int_{\R^d} \eta_R (y) \partial_i v^{\delta}(s,y)  u_\varepsilon(s,y)  dy \  \circ dB_{s}^{i}
\]
\[
+  \int_{0}^{t} \!\! \int_{\R^d} \partial_i\eta_R (y)  v^{\delta}(s,y)  u_\varepsilon(s,y)  dy \  \circ dB_{s}^{i}.
\]
 
Now for $\delta> 0$ fixed, passing to the limit as $\varepsilon$ goes to $0^+$, we obtain from the above equation 
\begin{equation}
\label{UNIQ20}
  \begin{aligned}
  \int_{\R^d}  u(t,X^\delta_t) & \; \eta_R (X_{t})  \varphi(x) \ dx  \\[5pt]
 & 
     = -\int_{0}^{t}\!\!    \int_{\R^d} \eta_R (y)   u(s,y)  \; div ( b^{\delta}(y) v^{\delta}(s,y))   \ dy  \  ds 
		\\[5pt]
 & -\int_{0}^{t} \!\! \int_{\R^d}  \eta_R (y)  \partial_i u(s,y) \; b^{i}(y) v^{\delta}(s,y)  \ dy \ ds
\\[5pt] &
+  \int_{0}^{t} \!\! \int_{\R^d} \partial_i\eta_R (y)  v^{\delta}(s,y)  u_\varepsilon(s,y)  dy \  \circ dB_{s}^{i}.
\end{aligned}
\end{equation}

Thus we deduce

\[
 \int_{\R^d}  u(t,X^\delta_t)  \; \eta_R (X_{t})  \varphi(x) \ dx  
\]

\[
=\int_{0}^{t}\!\!    \int_{\R^d} \eta_R (y)   \partial_i u(s,y)  \;  b^{\delta}(y) v^{\delta}(s,y)   \ dy  \  ds 
\]
\[
+ \int_{0}^{t}\!\!    \int_{\R^d} \partial_i \eta_R (y)   u(s,y)  \;  b^{\delta}(y) v^{\delta}(s,y)   \ dy  \  ds
\]
\[
-\int_{0}^{t} \!\! \int_{\R^d}  \eta_R (y)  \partial_i u(s,y) \; b^{i}(y) v^{\delta}(s,y)  \ dy \ ds
\]

\[
+ \int_{0}^{t} \!\! \int_{\R^d} \partial_i\eta_R (y)  v^{\delta}(s,y)  u(s,y)  dy \  \circ dB_{s}^{i}.
\]

Then, by   (\ref{est2}) and  applying the Dominated Convergence Theorem we
pass to the limit  as $\delta$ goes to $0^+$, to conclude that

\[
 \int_{\R^d}  u(t,X_t)  \; \eta_R (X_{t})  \varphi(x) \ dx  
\]

\[
= \int_{0}^{t}\!\!    \int_{\R^d} \partial_i \eta_R (y)   u(s,y)  \;  b(y) v(s,y)   \ dy  \  ds
\]

\begin{equation}\label{final}
+ \int_{0}^{t} \!\! \int_{\R^d} \partial_i\eta_R (y)  v(s,y)  u(s,y)  dy \  \circ dB_{s}^{i}.
\end{equation}

We observe that 

\[
 |\int_{0}^{t}\!\!    \int_{\R^d} \partial_i \eta_R (y)   u(s,y)  \;  b(y) v(s,y)   \ dy  \  ds|
\]
\[
 \leq C \big(|\int_{0}^{t}\!\!    \int_{R\leq|y|\leq 2R}   |u(s,y)|^{2p}  \; dy  \  ds|\big)^{\frac{1}{2p}}
\]

\noindent where used that $b$ is the linear growth, and

\[
 | \int_{\R^d} \partial_i\eta_R (y)  v(s,y)  u(s,y)  dy |
\]

\[
\leq \frac{1}{R} \big(| \int_{0}^{t} \int_{0}^{t}\!\!    \int_{R\leq|y|\leq 2R}   |u(s,y)|^{2p}  \; dy  \  ds|\big)^{\frac{1}{2p}}
\]

\noindent Passing to the limit as $R\rightarrow\infty$ we obtain  
\begin{equation}
\label{UNIQ30}
\int_{\R^d} u(t, X_t)  \varphi(x) dx =0
\end{equation}
for each $\varphi \in C_0^\infty(\R^d)$, and $t \in [0,T]$. 

\bigskip
 Thus , we have 
$$
  \begin{aligned}
    \!\! \int   \mathbb{E}| u(t,x) |^{2p}  \ dx &=   \!\! \int \mathbb{E}|  u(t, X_t(Y_t) ) |^{2p}  \ dx 
\\[5pt]
     &=   \;  \mathbb{E}   \!\!  \int JX_t    | \eta_R (X_{t}) u(t, X_t) |^{2p}  \ dx  =0
\\[5pt]  
\end{aligned}
$$
 where we have used \eqref{UNIQ30} and  that $X_t$ is a stochastic flows of diffeomorphism.  
Consequently, the thesis of our  theorem is proved. 
\end{proof}

\bigskip

\subsection{ Strong  Stability.}

To end up the well-posedness for the Cauchy problem 
\eqref{trasport}, it remains to show the stability 
 property for the solution with respect to the initial datum and respect to the drift term

\begin{proposition}\label{strong} Assume cond.  
Let $\{u_0^n\}$ be any sequence, with $u_0^n \in W^{1,2p}(\R^d)$ $(n \geq 1)$,  strong converging in  $W^{1,2p}$
to $u_0$. Let $u(t,x)$, $u^n(t,x)$ be the unique weak $W^{1,p}-$solution of the Cauchy problem \eqref{trasport},
for respectively the initial data $u_0$ and $u_0^{n}$. Then $u^n(t,x)$ strong converge to $u(t,x)$
in $L^{2p}(\Omega\times [0,T] \times \R^d)$ and  in $L^{p}(\Omega\times [0,T] , W^{1,p}(\R^d),\mu)$.
\end{proposition}

\begin{proof} By existence and uniqueness theorems we have  that

\[
u^n(t,x)=u_0^{n}(X_{t}^{-1})
\]

and 

\[
u(t,x)=u_0(X_{t}^{-1}).
\]

Then we obtain 

\[
\int_{0}^{T} \int_{\R^{d}} \E|u^n(t,x)-u(t,x)|^{2p} dx
\]

\[
=\int_{0}^{T} \int_{\R^{d}} \E| u_0^{n}(X_{t}^{-1})- u_0(X_{t}^{-1})|^{2p} dx ds
\]

\[
=\int_{0}^{T} \int_{\R^{d}} \E  | u_0^{n}(x)- u_0(x)|^{2p} |JX_t|  dx ds
\]

\begin{equation}\label{estab1}
\leq C  \int_{\R^{d}}  | u_0^{n}(x)- u_0(x)|^{2p} dx 
\end{equation}

where we used  (\ref{est1}) and (\ref{est3}).

Now, we have

\[
\int_{0}^{T} \int_{\R^{d}} \E|D[u^n(t,x)-u(t,x)]|^{p} e^{-|x|^{2}} dx ds
\]

\[
=\int_{0}^{T} \int_{\R^{d}} \E| Du_0^{n}(X_{t}^{-1})- Du_0(X_{t}^{-1})|^{p} |DX_{t}^{-1}|^{p} e^{-|x|^{2}}  dx ds
\]

\[
\leq   \big( \int_{0}^{T} \int_{\R^{d}} \E| Du_0^{n}(X_{t}^{-1})- Du_0(X_{t}^{-1})|^{2p}  e^{-|x|^{2}}  dx ds    \big)^{\frac{1}{2}}
\]

\[
  \times \big( \int_{0}^{T} \int_{\R^{d}}\E|DX_{t}^{-1}|^{2p}   e^{-|x|^{2}}  dx ds    \big)^{\frac{1}{2}}
\]

\[
\leq  C   \big( \int_{0}^{T} \int_{\R^{d}} \E| Du_0^{n}(x)- Du_0(x)|^{2p} \  JX_t \  e^{-|X_{t}|^{2}}  dx ds    \big)^{\frac{1}{2}}
\]

\begin{equation}\label{estab2}
\leq  C  \big(  \int_{\R^{d}}| Du_0^{n}(x)- Du_0(x)|^{2p}   dx    \big)^{\frac{1}{2}}
\end{equation}

where we used    Holder inqueality, (\ref{est1}) and  (\ref{est3}).

  Finally, from  estimation  (\ref{estab1}) and (\ref{estab2}) we conclude our proposition.

\end{proof}

\begin{proposition}\label{strong}Assume hypothesis \ref{hyp} for  $b$ and $b_{n}$ and $u_{0}\in C^{1,\alpha}(\R^{d})$. Moreover, we also assume 
$\|b_n -b\|_{C_b^{\theta}(\R^d,\R^d)}\rightarrow 0 $ as $n\rightarrow \infty$. Let $u(t,x)$, $u^n(t,x)$ be the unique weak $W^{1,p}-$solution of the Cauchy problem \eqref{trasport},
for respectively the drift  $b$ and $b_n$. Then $u^n(t,x)$ strong converge to $u(t,x)$
in $L^{p}(\Omega\times [0,T],W^{1,p}(\R^d),\mu)$.
\end{proposition}

\begin{proof}From  existence and uniqueness theorems we have  that

\[
u^n(t,x)=u_0(X_{t}^{-1,n})
\]

and 

\[
u(t,x)=u_0(X_{t}^{-1}).
\]

Then we obtain 

\[
\int_{0}^{T} \int \E|u^n(t,x)-u(t,x)|^{2p} e^{-|x|^{2}} dx ds
\]

\[
=\int_{0}^{T} \int \E| u_0(X_{t}^{-1})- u_0(X_{t}^{-1,n})|^{2p} e^{-|x|^{2}} dx ds
\]

\[
\leq \int_{0}^{T} \int \E  |X_{t}^{-1}- X_{t}^{-1,n}|^{\alpha}  e^{-|x|^{2}}   dx  ds.
\]

From (\ref{est2}) we conclude that $u^{n}\rightarrow u$ in $L^{p}(\Omega\times [0,T] \times \R^{d},\mu)$.

We get 
\[
\int_{0}^{T} \int \E|D[u^n(t,x)-u(t,x)]|^{p} e^{-|x|^{2}}  dx ds
\]

\[
=\int_{0}^{T} \int \E| Du_0(X_{t}^{-1,n})DX_{t}^{-1,n} - Du_0(X_{t}^{-1})DX_{t}^{-1}|^{p} e^{-|x|^{2}}   dx ds
\]

\[
\leq \int_{0}^{T} \int \E| Du_0(X_{t}^{-1,n})DX_{t}^{-1,n} - Du_0(X_{t}^{-1})DX_{t}^{-1,n}|^{p}  e^{-|x|^{2}}  dx  ds
\]

\begin{equation}\label{R2}
+ \int_{0}^{T} \int_{\R^{K}} \E| Du_0(X_{t}^{-1})DX_{t}^{-1,n} - Du_0(X_{t}^{-1})DX_{t}^{-1}|^{p}  e^{-|x|^{2}}  dx ds
\end{equation}

We have

\[
 \int_{0}^{T} \int \E| Du_0(X_{s}^{-1,n})DX_{t}^{-1,n} - Du_0(X_{s}^{-1})DX_{t}^{-1,n}|^{p}   e^{-|x|^{2}}   dx  ds
\]

\[
\leq  \big(\int_{0}^{T} \int (\E| Du_0(X_{s}^{-1,n}) - Du_0(X_{s}^{-1})|^{2p})  e^{-|x|^{2}} dx ds \big)^{\frac{1}{2}}
\]

\[
\times   \big(\int_{0}^{T} \int  |DX_{s}^{-1,n}|^{2p}  e^{-|x|^{2}}  dx ds \big)^{\frac{1}{2}}
\]

\[
\leq C  \big(\int_{0}^{T} \int (\E| X_{s}^{-1,n} -X_{s}^{-1}|^{2p\alpha} e^{-|x|^{2}} dx ds \big)^{\frac{1}{2}}
\]

where we used     (\ref{est1})  and
  (\ref{est3}) . From (\ref{est2}) we conclude that

\begin{equation}\label{R3}
 \int_{0}^{T} \int \E| Du_0(X_{t}^{-1,n})DX_{t}^{-1,n} - Du_0(X_{t}^{-1})DX_{t}^{-1,n}|^{p}  e^{-|x|^{2}}   dx \rightarrow 0
\end{equation}

as $n\rightarrow\infty$.

Now, we obtain 

\[
\int_{0}^{T} \int \E| Du_0(X_{t}^{-1})DX_{t}^{-1,n} - Du_0(X_{t}^{-1})DX_{t}^{-1}|^{p}   e^{-|x|^{2}}  dxds
\]

\[
\leq \big( \int_{0}^{T} \int \E| Du_0(X_{t}^{-1})|^{2p}    e^{-|x|^{2}}  dxds\big)^{\frac{1}{2}}
\]

\[
 \times  \big( \int_{0}^{T} \int (\E| DX_{t}^{-1,n} - DX_{t}^{-1}|^{2p}    e^{-|x|^{2}}  dxds \big)^{\frac{1}{2}} 
\]

\[
\leq \big( \int_{0}^{T} \int (\E| DX_{t}^{-1,n} - DX_{t}^{-1}|^{2p}    e^{-|x|^{2}}  dxds \big)^{\frac{1}{2}} 
\]

where we used (\ref{est1}) and  (\ref{est3}).  From (\ref{est1}) we deduce that

\begin{equation}\label{R4}
\int_{0}^{T}   e^{-|x|^{2}} \int  \E| Du_0(X_{t}^{-1})DX_{t}^{-1,n} - Du_0(X_{t}^{-1})DX_{t}^{-1}|^{p} ddx ds   dx\rightarrow 0
\end{equation}

as $n\rightarrow\infty$.

Finally from  (\ref{R2}), (\ref{R3}) and (\ref{R4}) we conclude our proposition.

\end{proof}

\section*{Acknowledgements}

Christian Olivera is partially supported by  CNPq
through the grant 460713/2014-0 and FAPESP by the grants 2015/04723-2 and 2015/07278-0.


\end{document}